\documentclass[12pt,a4paper]{article}

\usepackage{amsmath}
\usepackage{amsthm}
\usepackage{amssymb}
\usepackage{latexsym}
\usepackage{mathrsfs}
\usepackage[pdftex]{graphicx}

\newcommand{\D}{\mathcal D}
\newcommand{\e}{\varepsilon}

\newcommand{\N}{\mathbb{N}}
\newcommand{\n}{\eta}
\renewcommand{\o}{\omega}
\newcommand{\s}{\sigma}
\newcommand{\vfi}{\varphi}
\newcommand{\R}{\mathbb{R}}

\newcommand{\sn}{\s^\n}
\newcommand{\en}{\e^\n}
\newcommand{\un}{u^\n}

\renewcommand{\rm}{{\R^m}}
\renewcommand{\O}{\Omega}

\newcommand{\lsol}{L^{p}_{\textup{sol}}(\O)}
\newcommand{\lpsol}{L^{p'}_{\textup{sol}}(\O)}
\newcommand{\lpot}{L^p_{\textup{pot}}(\O)}
\newcommand{\vpot}{\mathcal V^p_{\textup{pot}}}
\newcommand{\vppot}{\mathcal V^{p'}_{\textup{pot}}}
\newcommand{\vsol}{\mathcal V^{p}_{\textup{sol}}}
\newcommand{\vpsol}{\mathcal V^{p'}_{\textup{sol}}}
\newcommand{\W}{{W_{\phantom 0}^{{\scriptscriptstyle -}1,p'}}}

\newcommand{\ds}{\displaystyle}

\renewcommand{\div}{\mbox{div}}

\newcommand{\curl}{\mbox{curl}}
\newcommand{\weakto}{\rightharpoonup}

\def\pref#1{(\ref{#1})}

\def\Xint#1{\mathchoice
   {\XXint\displaystyle\textstyle{#1}}%
   {\XXint\textstyle\scriptstyle{#1}}%
   {\XXint\scriptstyle\scriptscriptstyle{#1}}%
   {\XXint\scriptscriptstyle\scriptscriptstyle{#1}}%
   \!\int}
\def\XXint#1#2#3{{\setbox0=\hbox{$#1{#2#3}{\int}$}
     \vcenter{\hbox{$#2#3$}}\kern-.5\wd0}}
\def\dashint{\Xint-}

\newtheorem{lemma}{Lemma}

\newtheorem{theo}{Theorem}
\newtheorem{remark}{\mdseries{\itshape{Remark}}}
\newtheorem{example}{\mdseries{\itshape{Example}}}

\newenvironment{rem}{\begin{remark} \upshape}{\end{remark}}

\begin{document}

\title{Stochastic homogenization of subdifferential inclusions via scale integration}
\author{Marco Veneroni\footnote{Fakult\"at f\"ur Mathematik, Technische Universit\"at Dortmund, 44227 Dortmund, Germany. E-mail address: marco.veneroni@math.uni-dortmund.de}}

\date{November 11, 2010}
\maketitle

%
\normalsize

\begin{abstract}
We study the stochastic homogenization of the system
\begin{equation*}
	\left\{
	\begin{array}{l}
		\ds-\textup{div}\, \sn = f^\n \\
		\ds\sn \in \partial \phi_\n (\nabla \un),
	\end{array}
	\right.
\end{equation*}	
where $\phi_\n$ is a sequence of convex stationary random fields, with $p$-growth. We prove that sequences of solutions $(\sn,\un)$ converge to the solutions of a deterministic system having the same subdifferential structure. The proof relies on Birkhoff's ergodic theorem, on the maximal monotonicity of the subdifferential of a convex function, and on a new idea of scale integration, recently introduced by A. Visintin. 
\end{abstract}




\section{Introduction and main result}
In this paper we study the behaviour, as $\n \to 0$, of the solutions $\un(\cdot,\o),\sn(\cdot,\o)$ of the problem
\begin{equation}
\label{pizza}
	\left\{
	\begin{array}{ll}
		\ds-\textup{div}\, \sn(x,\o) = f^\n(x)  &\mbox{in }Q \times \O,\vspace{0.2cm}\\
		\ds\sn(x,\o) \in \partial \phi_\n (\nabla \un(x,\o),x,\o)\qquad &\mbox{in }Q \times \O,\vspace{0.2cm}\\
		\un(x,\o) = 0 &\mbox{on }\partial Q \times \O,
	\end{array}
	\right.
\end{equation}	
where $Q$ is a bounded domain in $\rm$ and $(\O,\mathcal F,\mu)$ is a probability space. The map $\xi \to \phi(\xi,\o)$ is convex for $\mu$-a.e. $\o \in \O$ and satisfies classical $p$-growth and coercivity conditions for $1<p<+\infty$. The \textit{oscillating} function $\phi_\n:\rm \times Q \times \O \to \R$ is defined as 
$$
	\phi_\n(\xi,x,\o):= \phi(\xi,T_{x / \n}\o),
$$	
for a dynamical system $T_x:\O \to \O$ on $\rm$.

\paragraph{A simple example.}
We refer to Section \ref{sec:prelergo} for the general definitions and results regarding probability, but in order to have a clearer picture of the setting we anticipate a simple one-dimensional example. 

Let $p=2$, $m=1$, $a_1,a_2,q_1,q_2\in \R$, $Q=(q_1,q_2)$, and consider a random function on $\R \times Q$ which in every point $x \in Q$ has probability $P\in (0,1)$ to have the form $\phi_1(\xi,x)=a_1\xi^2$ and probability $1-P$ to have the form $\phi_2(\xi,x)=a_2\xi^2$. The idea of homogenization is to approximate this random function by partitioning $Q$ into intervals of length $\n>0$, defining the random function $\phi_\n$ independently on each interval, and letting $\n$ tend to zero. Let 
$$M:=\left\{\o: \R\to \{a_1,a_2\}:\o \mbox{ is constant on every interval $(n,n\!+\!1)$, $n\in \mathbb N$}\right\},$$
then the dynamical system $T_x$ can be chosen as the shift operator
$$
\begin{array}{rclc}
	T_x: &L^\infty(\R) &\to 		&L^\infty(\R)\\
		& \o(\cdot)		&\mapsto	&\o(\cdot +x ),
\end{array}
$$		
and the set $\O$ is given by the set of all functions obtained from functions in $M$ by a shift $x \in \R$, i.e. $\O=M \times \mathbb T$, where $\mathbb T=\R/ \mathbb Z $ is the $1$-dimensional torus.
The probability measure on $\O$ is then the product measure of the measure induced by $P,1-P$ on $M$, times the Lebesgue measure on $\mathbb T$ (for further examples and details see, e.g., \cite[Section 7.3, ``Random Structure of Chess-board Type"]{JikovKozlovOleinik94} or \cite[pages 349--350]{DalmasoModica86a}). Finally, for all $\o$ which are continuous in 0, define
$$ \phi_\n(\xi,x,\o):=\o(x/\n)|\xi|^2.$$

Rescaling the parameter $x$ of the dynamical system by $1/\n$ corresponds, for vanishing $\n$, to a finer and finer mixing of the realizations of $\phi$. Under hypothesis of ergodicity for $T$ with respect to $\O$, it is then natural to expect the limit system to be deterministic. We prove that for almost every $\o \in \O$, sequences $(\sn,\un)$ of solutions of \pref{pizza} converge, as $\n \to 0$, to the solutions of a subdifferential inclusion governed by a deterministic convex function $\phi_0:\rm \to \R$.

Regarding the example above, in the case $P=1/2$, it is well known (see, e.g., \cite[page 350]{DalmasoModica86a}) that the limit function is given by
$$ \phi_0(\xi)=c|\xi|^2,$$
where $c=2a_1a_2/(a_1+a_2)$ is the harmonic mean of $a_1$ and $a_2$. 
 
 We turn to the main result of the paper. 
\begin{theo}
\label{th:nonlinear}
Let $(\O, \mathcal F, \mu)$ be a probability space with an ergodic dynamical system $T_x:\O \to \O$. Let $Q \subset \rm$ be a bounded domain, let $p\in (1,+\infty)$, $p'=p/(p-1)$, $(f^\n)_{\n>0} \subset W^{-1,p'}(Q)$, with $f^\n \to f$ strongly in $W^{-1,p'}(Q)$. 
Let $c_0,c_1,c_2>0$ and $\phi:\rm \times \O \to \R$ be such that
\begin{subequations}
\label{hyp:phi}
\begin{eqnarray}
	&\o \mapsto \phi(\xi,\o)   & \mbox{is $\mathcal F$-measurable }\forall\, \xi\in \rm\!,\label{hyp:meas}\\
	&\xi \mapsto \phi(\xi,\o) 		&\mbox{is convex for $\mu$-a.e.}\ \o \in \O,\qquad \quad\label{hyp:conv}\\
	& -c_0 +c_1|\xi|^p \leq \phi(\xi,\o) \leq c_0 + c_2|\xi|^p\quad & \mbox{for $\mu$-a.e.}\, \o \in \O, \forall\, \xi\in \rm\!.\label{hyp:growth}
\end{eqnarray}	
\end{subequations}
Define
$$ \phi_\n(\xi,x,\o):= \phi(\xi,T_{x/\n}\o),$$
then
\begin{itemize}
	\item[(i)] for every $\n>0$ there exists a couple $(\un,\sn)$ satisfying
$$ \un(\cdot,\o)\in W^{1,p}_0(Q),\qquad \sn(\cdot,\o) \in L^{p'}(Q)\qquad \mbox{for $\mu$-a.e. }\o\in \O,$$
\begin{equation}
\label{eq:problema}
\begin{array}{rlcl}
	-\textup{div}\, \sn\!\!\! & = f^\n& &\mbox{in }{W}^{-1,p'}(Q),\ \mbox{for $\mu$-a.e. } \o \in \O,\\
	\sn(x,\o) \!\!\!& \in \partial \phi_\n (\nabla \un(x,\o),x,\o)& &\mbox{a.e. in }Q \times \O,\\
	\un(x,\o)\!\!\! &= 0& &\mbox{a.e. on }\partial Q \times \O.
\end{array}
\end{equation}	
	\item[(ii)] For $\mu$-a.e. $\o \in \O$ the family of solutions $(\un(\cdot,\o),\sn(\cdot,\o))_{\n>0}$ is weakly compact in $W^{1,p}_0(Q)\times L^{p'}(Q)$. 	
	\item[(iii)] For $\mu$-a.e. $\o\in\O$, for every subsequence $(\n_k)_{k \in \N}$ such that 
\begin{align*}
	 \s^{\n_k} \weakto \s^*\qquad \mbox{weakly in }L^{p'}(Q),\\
	u^{\n_k} \weakto u^*\qquad \mbox{weakly in }W^{1,p}(Q),
\end{align*}	
the limit point $(\s^*, u^*)$ is a solution of
\begin{align}
	-\textup{div}\, \s^* &= f& &\mbox{in }Q, \nonumber\\
	\s^*&\in \partial \phi_0( \nabla u^*)& &\mbox{in }Q,\label{eq:inclusionzero}\\
	u^* &= 0& &\mbox{on }\partial Q.\nonumber
\end{align}	
\end{itemize}
The homogenized function $\phi_0:\rm \to \R$ is given by
\begin{equation}
\label{eq:phizero}
	\phi_0(\xi) 
			:=\inf \left\{ \int_\O \phi\big(\xi + u(\o),\o\big)\, d\mu:\quad u\in \lpot,\ \int_\O u\, d\mu =0\right\}.
\end{equation}
\end{theo}
(See Section \ref{sec:preweyl} for the definition of $\lpot$.)

\subsection{Literature}
We remark that the results of the present paper were already found by Messaoudi and Michaille in \cite{MessaoudiMichaille91, MessaoudiMichaille94}, using the method of epiconvergence of integral functionals, in the case of nonconvex integrands with growth $p\!>\!1$. This result was then generalized by Abddaimi, Michaille, and Licht \cite{AbddaimiMichailleLicht97}, to the case of quasiconvex integrands  with linear growth. As we detail in the next subsection, the main interest of the present paper is then to recover the known results with a new, simpler technique, which does not involve the convergence properties of integral functionals. In this section we review the main literature concerning stochastic homogenization.

The subject of homogenization has been widely studied since the early works of Babu\v{s}ka \cite{Babuska76}, De Giorgi and Spagnolo \cite{DegiorgiSpagnolo73}, Tartar \cite{Tartar74, Tartar77}, Bensoussan, Lions, and Papanicolau \cite{BensLionsPapa78},  Sanchez-Palencia \cite{SanchezPalencia80}. Restricting ourselves to the case of  stochastic homogenization (or \textit{homogenization in stationary ergodic media}), we recall that Kozlov \cite{Kozlov79} and Papanicolau and Varadhan \cite{Varadhan81, PapaVaradhan82} where the first to study second order linear elliptic PDEs. $\Gamma$-convergence techniques for the study of random integral functionals were employed by Dal Maso and Modica in \cite{DalmasoModica86b, DalmasoModica86a}. In the monograph \cite{JikovKozlovOleinik94} Jikov, Kozlov, and Oleinik collect results regarding (among the other topics) homogenization of linear elliptic problems, elliptic problems in perforated random domains and homogenization of random lagrangians with nonstandard growth conditions via $\Gamma$-convergence, with applications to elasticity. 

More recently, the mathematical analysis community showed a renewed interest in stochastic homogenization of different nonlinear problems. As examples we cite the papers on parabolic operators \cite{Svanstedt07, Svanstedt08}, the works on Hamilton-Jacobi equations \cite{LionsSouganidisTA}, on fully nonlinear elliptic equations \cite{CaffarelliSouganidisWang05, CaffarelliSouganidis10}, one-dimensional plasticity \cite{Schweizer09}, and on variance estimates for the homogenized coefficients of discrete elliptic equations \cite{GloriaOttoTA}. We also mention a first extension of Nguetseng's two-scale convergence \cite{Nguetseng89} to the stochastic setting \cite{BourgeatMikelicWright94}.

We note that problem \pref{eq:problema} was studied in \cite{Damlamianetal08} in the setting of periodic homogenization by means of Mosco-convergence and periodic unfolding. In the proof we employ a different method of homogenization but we adopt the same variational setting and convex analysis tools.

\subsection{Discussion}

As we noted above, the form of the homogenized function $\phi_0$ is already well-known in the literature. The real interest of this paper lies rather in a new proof which is direct, in the sense that it does not make use of $\Gamma$-convergence, $G$-convergence, epiconvergence, or \textit{two-scale} convergence, showing instead simultaneous convergence of the two components of the solution $(\sn,\un)$. The proof is also completely self-contained, taken into account that some probability tools like ergodicity (sometimes traded for \textit{independence at large distances}, as in \cite{DalmasoModica86a}) are intrinsecally part of the problem. Finally, we apply to the stochastic setting a new idea of homogenization introduced by Visintin in \cite{Visintin09} for the periodic setting (under more general hypothesis). This method, named ``scale integration," deals with the integration with respect to $y$ of the relation
$$ w(x,y) \in \partial \vfi\big(v(x,y),x,y\big)\qquad \mbox{on }Q \times Y,$$ 
where $\partial \vfi(\cdot,x,y)$ is a cyclically monotone mapping on a separable real Banach space  $X$, $Q\subset \rm$, and $Y$ is the $m$-dimensional unit torus. Relying on the splitting of the space-regularity that is at the basis of compensated compactness, and on the properties of Fenchel-Legendre transform, the method yields
$$ \int_Y w(x,y)\,dy \in \partial \vfi_0 \left(\int_Y v(x,y)\, dy ,x\right)\qquad \mbox{on }Q,$$ 
where $\vfi_0$ is the solution of a minimization problem in a suitable subspace $V\subset X$
\begin{equation*}
	\vfi_0(\xi,x) :=\inf_{u \in V} \int_Y \vfi\big(\xi + u(y),x,y\big)\, dy.
\end{equation*}

\subsubsection{Description of the proof}
The main ideas underlying the proof of Theorem \ref{th:nonlinear} can be summarized as follows 
\begin{itemize}
	\item We study the auxiliary problem  
	$$ -\div\, s=0,\qquad s \in \partial \phi (\nabla w),$$
	set in the probability space only. Weyl's decomposition of $L^2(\O)$ spaces into potential and solenoidal field provides the correct functional setting for the variational formulation and allows for a rigorous definition of the derivatives of $s$ and $w$. Convexity of $\phi$ ensures the existence of a solution.
	\item By Birkhoff's ergodic Theorem \ref{th:birkhoff}, for a.e. $\o \in \O$ the realizations 
	$$s_\n(x)=s(T_{x/\n}\o),\quad \nabla w_\n(x)=\nabla w(T_{x/\n}\o)$$ 
	converge to their expected values 
	$$E(s)=\int_\O s\, d\mu,\quad E( \nabla w)=\int_\O \nabla w\, d\mu.$$
	\item Exploiting the scale integration method, we can relate the expected values by the \textit{homogenized function} $\phi_0$ defined in \pref{eq:phizero}
	$$ E(s)\in \partial \phi_0(E(\nabla w)).$$
	\item By compensated compactness (see Lemma \ref{lemma:divcurl}) we can pass to the limit in the monotonicity inequality
	$$ (\sn-s_\n,\nabla \un - \nabla w_\n)\geq 0$$
	and recover \pref{eq:inclusionzero} by maximal monotonicity of $\partial \phi_0$.
\end{itemize}

\subsubsection{Further questions}
We remark that in some of the steps above we use tools that allow for more general hypothesis than the ones assumed in this work. It would be interesting to understand if the scale integration method could be applied to recover the same general results which can be obtained via $\Gamma$- or epiconvergence, and if it could yield new results. We present a few examples of directions of extension.

\paragraph{Growth $p=1$.} All the properties which rely only on convexity and monotonicity are still valid, but since $L^1$ is not reflexive, problems will arise, e.g., from the lack of compactness of $W^{1,1}$, which leads to a setting in spaces of bounded variation. This problem was solved via epiconvergence, see \cite{AbddaimiMichailleLicht97}.

\paragraph{Nonconvex $\phi$.} The scale integration method does not require the integrand $\phi$ to be convex, so in principle it could be applied also in this case. A number of questions arise nonetheless, starting from the lack of lower-semicontinuity for the related integral functional. See the results in \cite{MessaoudiMichaille91, MessaoudiMichaille94}.

\paragraph{Time-dependent problems.}
One of the main physical applications of equations \pref{eq:problema} is to problems of nonlinear viscosity like
\begin{equation*}
\begin{array}{rlcl}
	-\textup{div}\, \s\!\!\! & = f,\\
	 \partial_t (\nabla^s u)\!\!\!& \in \partial \phi \big(\s\big).
\end{array}
\end{equation*}	
Since scale integration applies also to time-dependent problems, we hope to be able to use the present result as a starting point for the study of this kind of problems in the future. A similar use of an auxiliary problem and oscillating test function method was used, in the case of periodic homogenization, in \cite{SchweizerVeneroniTA}.

\subsubsection{Plan of the paper}
The paper is organized as follows. Section \ref{sec:preliminaries} contains all the notation, definitions, and the main results regarding ergodicity, Weyl's decomposition, compensated compactness, and convex analysis. The proof of Theorem \ref{th:nonlinear} is given in Section \ref{sec:proof}.

\section{Preliminaries}
\label{sec:preliminaries}

Notation: for all open sets $E \subset \rm$, let $\D(E)$ be the set of all infinitely differentiable functions with compact support in $E$, and let $\D'(E)$ be its dual space, that is, the distributions on $E$.

\subsection{Ergodicity}
\label{sec:prelergo}
For this part of classical theory we follow the exposition in \cite[Section 7]{JikovKozlovOleinik94}, to which we refer for further details and examples.

Let $(\O,\mathcal F,\mu)$ be a probability space, with $\s$-algebra $\mathcal F$ and probability measure $\mu$. An \textit{m-dimensional dynamical system} is a family of mappings
$$ T_x : \O \to \O\qquad \forall\,x\in \R^m,$$
satisfying
\begin{enumerate}
	\item \textit{Group property}. $T_0= I$ (where $I$ is the identity),
	$$ T_{(x+y)}=T_x\,T_y\qquad \forall\, x,y \in \R^m.$$
	\item \textit{Mass invariance}. The mappings $T_x: \O \to \O$ preserve the measure $\mu$ on $\O$, that is, for every $x\in \R^m$ and every measurable set $E \in \mathcal F$ we have
	$$ T_x E\in \mathcal F,\qquad \mu(T_x E)=\mu(E).$$
	\item \textit{Measurability}. For any measurable function $f: \O \to \R^n$, the function $\tilde f:\R^m \times \O  \to \R^n$ defined as  $\tilde f(x,\o):=f(T_x \o)$ is also measurable.
\end{enumerate}	 
A \textit{random field on} $\R^m$ is a mapping $\xi:\R^m \times \O \to \R^n$ which associates to every $x\in \R^m$ a random variable $\xi(x,\cdot)$ with values in $\R^n$. A random field is said to be \textit{stationary} if for any finite set $x_1,\ldots,x_k \in \R^m$ and any $h \in \R^m$ the distribution of the random vector
$$ \xi(x_1+h),\ldots,\xi(x_k +h) $$
does not depend on $h$. In particular, $\xi$ is stationary if it can be represented in the form
\begin{equation}
\label{def:stationary}
	\xi(x,\o)=f(T_x\o),
\end{equation}	
where $f$ is a fixed random variable on $\O$, and $T$ is a dynamical system. The converse statement is analyzed, e.g., in \cite{Doob90}.

Let a dynamical system $T$ be given, a measurable function $f$ defined on $\O$ is said to be \textit{invariant} if 
$$ f(T_x \o) = f(\o)\qquad \forall\,(x,\o) \in \R^m,$$
almost surely in $\O$. A dynamical system $T$ is said to be \textit{ergodic} if every invariant function is constant almost surely in $\O$.

Let $\o \mapsto f(\o)$ be a measurable function on $\O$. For a fixed $\o \in \O$ the function $x \mapsto f(T_x\o)$ is said to be a \textit{realization of $f$}. We often denote the realization of $f$ by $\tilde f(\cdot)$.
\medskip

Let $p\geq 1$, then we denote by $L^p(\O)$ the usual space formed by the equivalence classes of $\mathcal F$-measurable functions $f:\O \to \rm$, such that
$$ \int_\O |f(\o)|^p\, d\mu(\o) < +\infty,$$
and by $L^\infty(\O)$ the space of measurable essentially bounded functions.

\begin{lemma}
\label{lemma:lp}
The following properties hold:
\begin{itemize}
	\item if $f \in L^p(\O)$, $p \in [1,+\infty[$, then $\mu$-almost all realizations of $f$ belong to $L^p_{loc}(\rm)^m$.
	\item if $f_n \to f$ strongly in $L^p(\O)$, then there exists a subsequence $(f_{n_k})_{k \in \N}$ such that for $\mu$-a.e. $\o\in \O$ we have $\tilde f_{n_k} \to \tilde f $ strongly in $L^p_{loc}(\rm)^m$, where $\tilde f_{n_k} (x) =f_{n_k}(T_x \o),$ $\tilde f(x)= f(T_x \o)$.
\end{itemize}	
\end{lemma}
\begin{proof}
Using the measure-preserving property of $T$, for all $x\in \rm$ 
\begin{align*}
    \|f_n - f\|^p_{L^p(\O)}&= \int_\O |f_n(\o)-f(\o)|^p\, d\mu(\o)\\
        &=\int_\O |f_n(T_x\o)-f(T_x\o)|^p\, d\mu(\o).
\end{align*}
Let $K \subset \subset \rm$, then by the measurability property of $T$ and Fubini's theorem
\begin{align*}
    \|f_n - f\|^p_{L^p(\O)}    &= \frac{1}{|K|}\int_K\int_\O |f_n(T_x\o)-f(T_x\o)|^p\, d\mu(\o)\,dx\\
        &= \frac{1}{|K|}\int_\O \int_K|f_n(T_x\o)-f(T_x\o)|^p\,dx\, d\mu(\o).
\end{align*}        
This implies that $\mu$-almost all realizations of $f$ belong to $L^p_{loc}(\rm)^m$, and since $\|f_n - f\|^p_{L^p(\O)} \to 0$, we can find a subsequence $n_k$ such that 
$$F_{n_k}(\o):=\int_K|f_{n_k}(T_x\o)-f(T_x\o)|\,dx$$
converges to zero $\mu$-almost everywhere.
\end{proof}

Let $f\in L^1_{loc}(\R^m)^m.$ A number $M\{f\}$ is called the \textit{mean value} of $f$ if
$$ \lim_{\n \to 0} \frac{1}{|K|} \int_K f\left( \frac x\n \right)\, dx = M\{f\}$$	
for any measurable bounded set $K \subset \R^m$. Denoting $ K_\n:=\{x \in \R^m: \n x \in K\},$
the mean value of $f$ can also be written, after rescaling, as
$$ \lim_{\n \to 0 } \frac{\n^m}{|K|} \int_{K_\n} f (x)\, dx = M\{f\}.$$	
The \textit{expected value} of a random variable $f$ on $\O$ is defined as
$$ E(f):=\int_\O f(\o)\, d\mu.$$

\begin{theo}\textup{(Birkhoff's Ergodic Theorem, \cite[Lemma 15.1]{JikovKozlovOleinik94})}
\label{th:birkhoff}
Let $f\in L^1(\O)$. Then, for a.e. $\o \in \O$, the mean value of the realization $f(T_x\o)$ exists and it satisfies
$$ \int_\O M\{f(T_x\o)\}\, d\mu=E(f) .$$
If the system $T$ is ergodic, then the mean value does not depend on $\o$ almost surely and it satisfies
$$ M\{f(T_x \o)\} = E(f)\qquad \mbox{for $\mu$-a.e.\,}\o\in\O.$$
\end{theo}

\begin{rem}
Birkhoff's theorem tells us that almost every realization $\tilde f_\n(x)=f(T_{x/\n}\o)$ satisfies
$$ \lim_{\n \to 0} \dashint_K \tilde f_\n(x)\, dx = E(f),$$
and it is important to notice that this holds for every measurable bounded set $K \subset \rm$. This entails in particular that if $f\in L^p(\O)$, then
$$ \tilde f_\n \weakto E(f)\qquad \mbox{weakly in }L^p_{loc}(\rm)^m.$$
\end{rem}

\subsection{Div-curl Lemma}
Compensated compactness is the fundamental idea for passing to the limit as $\n \to 0$ in a product of weakly converging functions $\sn,\en$.

If $\s,\e\in L^2(Q)^m,$ and $\{\sn\}$, $\{\en\}$ are two sequences of vector fields in $L^2(Q)^m$ such that
$$ \sn \weakto \s,\qquad \en \weakto \e\qquad \mbox{weakly in }L^2(Q)^m,$$
it is in general false that
$$ \sn \cdot \en \weakto \s \cdot \e,$$
in any sense. Consider, for example, $\sn(x)=\en(x)=\sin(x/\n)$ in $L^2(0,1)$. The theory of compensated compactness, introduced by L. Tartar and F. Murat in \cite{Tartar77, Murat78, Tartar79}, gives general conditions on the derivatives of $\sn,\en$, and on a function $F$ in order to conclude that
$$ F(\sn, \en) \weakto F(\s, \e)$$
in the sense of distributions. The following statement is one of the consequences of compensated compactness theory.
\begin{lemma}[Div-Curl lemma, \cite{Murat81}]
\label{lemma:divcurl}
Let $p\in (1,+\infty)$, $p':=p/(p-1)$. Let $\sn,\s \in L^{p'}(Q)^m$ and $\en,\e \in L^p(Q)^m$  such that
$$ \sn \weakto \s\quad \mbox{weakly in }L^{p'}(Q)^m, \qquad \en \weakto \e\quad \mbox{weakly in }L^p(Q)^m.$$
In addition, let $f\in W^{-1,p'}(Q)$ and assume that
\begin{align*}
	&\textup{curl}\, \en=0,\\
	&\textup{div}\, \sn\to  f\qquad \mbox{strongly in }W^{-1,p'}(Q).
\end{align*}
Then	
$$ \sn \cdot \en \stackrel{*}{\weakto} \s \cdot \e\qquad \mbox{in }\D'(Q).$$	
\end{lemma}

\subsection{Weyl's decomposition}
\label{sec:preweyl}

Weyl's decomposition theorem, or more correctly ``Peter-Weyl's" theorem, was originally stated for compact topological groups \cite{PeterWeyl27}. In the context of stochastic homogenization, it is used to provide an orthogonal decomposition of $L^2(\O)$ into functions, the realizations of which are vortex-free and divergence-free, in the sense of distributions  (see, e.g., \cite[Lemma 7.3]{JikovKozlovOleinik94}). For $p \in [1,+\infty[$, Weyl's theorem can be generalized to a relation of orthogonality between subspaces of $L^p(\O)$ and $L^{p'}(\O)$.

Let $v=(v^1,\dots,v^m),$ $v^i\in L^p_{loc}(\rm)$, $i=1,\dots,m$, be a vector field. We say that $v$ is \textit{potential} (or \textit{vortex-free}) in $\rm$ if
\begin{equation}
\label{def:potential}
 	\int_\rm v^i(x)\frac{\partial\vfi(x)}{\partial x_j} - v^j(x)\frac{\partial\vfi(x)}{\partial x_i}\, dx=0\qquad \forall\,i,j=1,\ldots,m,\qquad \forall\, \vfi \in \D(\rm).
\end{equation}
Recall that every vortex free vector field in a simply connected domain is a gradient of some function $u$. In $\rm$ it means that there exists $u\in W^{1,p}_{loc}(\rm)$ such that $v=\nabla u$. We say that a vector field $v$ is \textit{solenoidal} (or \textit{divergence-free}) in $\rm$ if
\begin{equation}
\label{def:solenoidal}
	\sum_{i=1}^m \int_\rm v^i(x)\frac{\partial\vfi(x)}{\partial x_i}\, dx=0\qquad \forall\, \vfi \in \D(\rm).
\end{equation}	
\medskip

Now let us consider vector fields on the probability space $(\O, \mathcal F,\mu)$. A vector field $f \in L^p(\O)$ is called \textit{potential} if $\mu$-almost all its realizations $x\mapsto f(T_x \o)$ are potentials, in the sense of \pref{def:potential}. The space of potential vector fields on $\O$ is denoted by $\lpot$. A vector field $f \in L^p(\O)$ is called \textit{solenoidal} if $\mu$-almost all its realizations $x\mapsto f(T_x \o)$ are solenoidal, in the sense of \pref{def:solenoidal}. The space of solenoidal vector fields on $\O$ is denoted by $\lsol$.

\begin{lemma}  
\label{lemma:closed}
 The spaces $\lpot,\lsol$ satisfy
\begin{eqnarray}
	\lpot \mbox{ and } \lsol  \mbox{ are \textit{weakly closed subspaces} of }L^p(\O).\label{eq:weaklyclose}\\
	E(\s \cdot \e)=E(\s)\cdot E(\e)\qquad \quad\mbox{for all }\s \in \lpsol,\ \e \in \lpot.
\end{eqnarray}	
\end{lemma}
\begin{proof}
Let us prove \pref{eq:weaklyclose} for $\lsol$, the proof for $\lpot$ is identical. Let $(u_n)\subset \lsol$, and assume that $u_n \to u$, strongly in $L^p(\O)$. Then, by definition of $\lsol$, $\mu$-almost every realization $\tilde u_n$ satisfies
\begin{equation}
\label{eq:unsol}
	\sum_{i=1}^m \int_\rm \tilde u_n(x)\frac{\partial\vfi(x)}{\partial x_i}\, dx=0\qquad \forall\, \vfi \in \D(\rm).
\end{equation}	
By Lemma \ref{lemma:lp} there exists a subsequence $(n_k)$ such that $\tilde u_{n_k} \to \tilde u$, strongly in $L^p_{loc}(\O)$. Passing to the limit as $k\to \infty$ in \pref{eq:unsol} yields $u\in \lsol$, and therefore $\lsol$ is strongly closed in $L^p(\O)$. Being $\lsol$ a linear subspace of $L^p(\O)$, strongly closed is equivalent to weakly closed.
\end{proof}

Set
\begin{align*}
	\vpot = \big\{ f \in \lpot\ :\ E(f)=0 \big\},\\
	\vsol = \big\{ f \in \lsol\ :\ E(f)=0 \big\}.
 \end{align*}
 It holds
 $$ \lpot = \vpot \oplus \rm,\qquad \qquad \lsol = \vsol \oplus \rm.$$

\begin{lemma} \textup{\cite[lemma 15.1]{JikovKozlovOleinik94}}
\label{lemma:weylp}
Let $p \in [1,+\infty)$, $p'=p/(p-1)$. The relations
$$
	(\vpot)^\perp = \vpsol \oplus \rm,\qquad (\vsol)^\perp = \vppot \oplus \rm
$$
hold in the sense of duality between the spaces $L^p(\O)$ and $L^{p'}(\O)$.	
\end{lemma}

\begin{rem}
If $v\in \lpot$, as noted above, for $\mu$-a.e. $\o \in \O$ we can find a function $u$ such that $\nabla_x u(x,\o) =v(T_x \o)$, but this does not imply that $u=u(T_x \o)$, i.e., that $u$ is a \textit{stationary} random field. In particular, we are not allowed to apply Birkhoff's theorem to the sequence $u(x/\n,\o)$.
\end{rem}

\subsection{A review of convex analysis}

We review some basic facts of convex analysis. We recall the statements
in the case when $(X,| \cdot |_X)$ is a reflexive Banach space, with dual space $(X',| \cdot |_{X'})$ and duality pairing ``$\left<\cdot,\cdot \right>$'', having in mind the applications $X=\rm,\ X=L^p(\O)$. 

Let $\vfi:X \to \R \cup \{+\infty\}$,  the domain of $\vfi$ is
$$ dom(\vfi):=\left\{ \e \in X: \vfi(\e)<+\infty\right\}.$$
If $\vfi$ is proper (i.e., $dom(\vfi)\neq \emptyset$), the Legendre--Fenchel transform (or Young--Fenchel transform, or conjugate function) $\vfi^*$ is given by
\begin{equation}
\label{def:conj}
  \vfi^*: X \to \R \cup \{+\infty\},\quad \s \ds
  \mapsto \sup_{\e \in X} \{\left<\s,\e\right> - \vfi(\e) \}.
\end{equation}
The growth of the conjugate function $\vfi^*$ is related to the growth of $\vfi$ by the following result.
\begin{lemma}
\label{lemma:stargr}
Let $c_0,c_1,c_2>0$, $p\in[1,+\infty]$, $p'=p/(p-1)$. If $\vfi$ satisfies
$$-c_0 +c_1|\e|_X^p \leq \vfi(\e) \leq c_0 + c_2|\e|_X^p\qquad   \forall\, \e\in X,$$
then there exist $\bar c_1,\bar c_2>0$ such that $\vfi^*$ satisfies
\begin{equation}
\label{eq:p1growth}
	-c_0 +\bar c_1|\s|_{X'}^{p'} \leq \vfi^*(\s) \leq c_0 + \bar c_2|\s|_{X'}^{p'}\qquad   \forall\, \s\in X'.
\end{equation}	
\end{lemma}
The subdifferential of $\vfi$, computed in the point $\e \in dom(\vfi)$, is the set
\begin{equation*}
  \partial \vfi(\e) =  \left\{ \s \in X' \mbox{ such that }  \vfi(\xi) \geq
    \vfi(\e) + \left<\s,\xi - \e\right>  \quad \forall\, \xi \in X\right\}.
\end{equation*}
Some useful properties of convex functions are summarized in the
following lemma, for a proof we refer to \cite{EkelandTemam74}.

\begin{lemma}\label{lemma:convprop}
  Let $\vfi:X \to \R \cup {+\infty}$ be proper, then, $\forall\,\e \in X$, $\forall\,\s \in X'$,
  \begin{itemize}
  	\item[(i)] $\vfi^*$ is convex, lower-semicontinuous, and  $dom(\vfi^*)\neq \emptyset$.
  	\item[(ii)] $\vfi(\e) + \vfi^*(\s) \geq \left<\s,\e\right>$.
  	\item[(iii)] $\s  \in \partial \vfi(\e)\ \Rightarrow\ \e \in \partial \vfi^*(\s)$.
   	\item[(iv)] $[\vfi(\e)=\vfi^{**}(\e),\ \e \in \partial \vfi^*(\s)]\ \Rightarrow \s \in \partial \vfi(\e)$.
  	\item[(v)] $ \e \in \partial \vfi(\s)\ \ \Leftrightarrow\ \vfi(\e) + \vfi^*(\s) = \left<\s,\e \right>$.
  \end{itemize}
 \end{lemma}
  The equality in (v) is also known as Fenchel's equality, while (ii) is
  referred to as (Young--)Fenchel's \emph{in}equality. Note that for a proper, convex, lower-semicontinuous function $\vfi$, implication (iii) becomes an equivalence.

Define the orthogonal complement $V^\perp$ of a closed subspace $V \subset X$ as
$$ V^\perp:=\left\{ \s \in X': \forall\, \e \in V,\ \left<\s, \e\right> =0\right\}.$$
\begin{theo}[Fenchel-Rockafellar]
\label{th:rock}
Let $V \subset X$ be a closed subspace, let $\vfi$ be proper, convex, lower-semicontinuous. Assume that there is a point of $V$ where $\vfi$ is continuous, then
$$ \inf_{\e \in V}\vfi(\e) + \inf_{\s \in V^\perp}\vfi^*(\s)=0.$$
\end{theo}
We conclude this section extending these concept to integral functionals. Note that if $\vfi:\rm \times \O \to \R$ satisfies assumptions \eqref{hyp:phi}, then for all $u\in L^p(\O)$ the map
$ \o \mapsto \vfi(u(\o),\o)$ is measurable and 
$$\Phi(u):=\int_\O \vfi(u(\o),\o)\,d\mu < \infty.$$
We can now connect integral functionals and conjugate functions with the following theorem.
\begin{theo}[Rockafellar \cite{Rockafellar68}]
\label{th:rock2}
Let $\vfi:\rm \times \O \to \R$ satisfy assumptions \eqref{hyp:phi}, then 
$$ \Phi^*(z)=\int_\O \vfi^*(z(\o),\o)\,d\mu.$$
\end{theo}

\section{Proof of Theorem \ref{th:nonlinear}}
\label{sec:proof}

\begin{proof}
The proof is divided into three steps.

\paragraph{Step I.} The auxiliary problem.  
Let $\xi\in \rm$ be fixed, we look for functions $v\in L^p(\O)$, $z \in L^{p'}(\O)$ such that
\begin{equation}
\label{eq:auxnonlin}
	v\in \vpot,\quad z \in \lpsol\qquad \mbox{with}\quad z\in \partial \phi(\xi + v).
\end{equation}
By Lemma \ref{lemma:convprop}-(iii) and (iv), \pref{eq:auxnonlin} is equivalent to
\begin{equation}
\label{eq:auxnonlind}
	v\in \vpot,\quad z \in \lpsol\qquad \mbox{with}\quad \xi + v \in \partial \phi^*(z).
\end{equation}
Define the functionals
	\begin{eqnarray}
		F:L^p(\O)\to \R, \quad& &\ds F(u):= \int_\O \phi(\xi + u(\o),\o)d\mu,  \vspace{0.2cm}\label{def:F}\\
		G:L^{p'}(\O) \to \R, \quad  & &\ds G(y):= \int_\O \phi^*(y(\o),\o) - \xi\cdot y(\o)\,d\mu. \label{def:G}
	\end{eqnarray}
We claim that \pref{eq:auxnonlin} (or \pref{eq:auxnonlind}) is equivalent to the minimization problem:
\begin{subequations}
	\label{aux:min}
	\begin{eqnarray}
		\ds v \mbox{ solves:}&\ds \min_{u \in \vpot}F(u)   \vspace{0.2cm}\label{aux:mindir}\\
		and & &\vspace{0.2cm}\nonumber\\
		\ds z \mbox{ solves:}&\ds \min_{y \in \lpsol} G(y). \label{aux:mindual}
	\end{eqnarray}
\end{subequations}
Let us show that \pref{eq:auxnonlin} implies \pref{aux:mindir}. For $f \in L^1(\O)$, let $\hat f:= f - E(f)$ be the \emph{fluctuation} of $f$, so that 
$$E(\hat f) =0,\qquad f = \hat f + E(f).$$ 
By definition of subdifferential, $z \in \partial\phi(\xi +v)$ if and only if
$$ z\cdot(\lambda - (\xi +v)) + \phi(\xi + v)\leq \phi(\lambda)\qquad \forall\,\lambda \in \rm.$$
Then, for all $u \in \lpot$ such that $E(u)=\xi$, we have, $\mu$-a.e. in $\O$,
$$ z\cdot(\xi + \hat u - (\xi +v)) + \phi(\xi + v)\leq \phi(\xi + \hat u).$$
By Weyl's Lemma \ref{lemma:weylp}, it holds 
$$ \int_\O z(\o) \cdot \hat u(\o)\, d\mu = 0 = \int_\O z(\o) \cdot v(\o)\,d\mu,$$
and we conclude
$$ \int_\O \phi(\xi + v(\o),\o)\, d\mu \leq \int_\O \phi(\xi + \hat u(\o),\o)\, d\mu$$
for all $\hat u \in \lpot$ such that $E(\hat u)=0$, i.e., for all $\hat u \in \vpot$. The proof that \pref{eq:auxnonlind} implies \pref{aux:mindual} follows in the same way.

Let us prove the converse implication. By definition of $F$ \pref{def:F} and definition of conjugate function \pref{def:conj}, for all $y\in L^{p'}(\O)$ it holds
\begin{align*} 
	F^*(y) &= \sup_{u \in L_{\phantom *}^p(\O)}\left\{  \left< y ,  u\right> - F(u)\right\}\\
			& = \sup_{u \in L_{\phantom *}^p(\O)}\left\{  \int_\O y\cdot u - \phi(\xi +u)\, d\mu \right\}\\
			&= \sup_{v \in L_{\phantom *}^p(\O)}\left\{  \int_\O y\cdot (v -\xi) -\phi(v)\, d\mu \right\}\\
			&=\sup_{v \in L_{\phantom *}^p(\O)}\left\{   \left<y, v \right> -\int_\O\phi(v)\, d\mu \right\} - \int_\O y\cdot \xi\,d\mu\\
			&\stackrel{Th. \ref{th:rock2}}{=}\int_\O \phi^*(y) - y\cdot\xi\, d\mu=G(y).
\end{align*}	
Let now $\xi\in \rm$ and let $(v,z)$ be a solution of \pref{aux:min}. Since by Weyl's theorem $(\vpot)^\perp = \lpsol$, by Theorem \ref{th:rock} we obtain
\begin{equation}
\label{eq:userock}
\begin{split}
	\int_\O \phi(\xi +v) = \min_{u \in \vpot}F(u) \stackrel{Th. \ref{th:rock}}{=} &-\min_{y\in \lpsol}F^*(y)\\
	=&-G(z)= \int_\O \xi\cdot z - \phi^*(z)= \int_\O (\xi+v)\cdot z - \phi^*(z).
	\end{split}
\end{equation}
Since by Lemma \ref{lemma:convprop}-(ii)
$$ \phi(\xi +v) \geq (\xi+v)\cdot z - \phi^*(z)\quad \mbox{for $\mu$-a.e. }\o \in \O,$$
by \pref{eq:userock} we conclude that
$$ \phi(\xi +v)  + \phi^*(z) =  (\xi+v)\cdot z  \quad \mbox{for $\mu$-a.e. }\o \in \O,$$
and therefore, by \pref{lemma:convprop}-(v), that $(v,z)$ solves \pref{eq:auxnonlin}. 
\medskip

We turn now to the existence of a solution to the minimization problems \pref{aux:mindir}-\pref{aux:mindual}. We note that
\begin{itemize}
	\item by Lemma \ref{lemma:closed} the subspaces $\lsol$ and $\vpot$ are  closed with respect to the \emph{weak} convergence in $L^p(\O)$. 
	\item  $\phi,\phi^*$ are convex functions, with growth \pref{hyp:growth} and \pref{eq:p1growth} respectively, which implies that the functionals $F$ and $G$ defined in \pref{def:F} and \pref{def:G} are coercive and lower-semicontinuous with respect to the \emph{weak} convergence in $L^p(\O)$. 
\end{itemize}
Then, any minimizing sequence for $F$ ($G$) is weakly compact in $L^p(\O)$ (resp. $L^{p'}(\O)$) and, owing to lower-semicontinuity, any weak limit point is a solution of \pref{aux:mindir} (resp. \pref{aux:mindual}). 
\begin{rem}
If $\phi$ and $\phi^*$ are strictly convex, then the solution of \pref{aux:min} is unique. Under the hypothesis of Theorem \ref{th:nonlinear} problem \pref{aux:min} might have more than one solution.
\end{rem}

\paragraph{Step II.} The homogenized functional $\phi_0$.  Let us define  $\phi_0:\rm \to \R$
\begin{equation}
\label{def:phi0}
	\phi_0(\xi):=\inf_{u\in \vpot}\int_\O \phi(\xi + u(\o),\o)\, d\mu.
\end{equation}
We claim that 
\begin{equation}
\label{eq:equalityhomo}
\begin{split}
	\mbox{for all $\xi \in \rm$, every solution $(v,z)$ of problem \pref{eq:auxnonlin} satisfies}\\
	E(z) \in \partial \phi_0(\xi).\hspace{5cm}
\end{split}	
\end{equation}	
The proof is a direct translation to functions in $L^p(\O)$ of the proof of Theorem 3.2 in \cite{Visintin09}.
\medskip

\begin{proof}[Proof of \pref{eq:equalityhomo}] Recall that $E(f)$ denotes the expected value $\int_\O f\,d\mu$ of a random function $f$. 
Define $\psi_0 :\rm \to \R$
\begin{equation}
\label{def:psi0}
	\psi_0(\zeta):= \inf\left\{ \int_\O \phi^*(z(\o),\o)\,d\mu\ :\quad z \in \lpsol,\ E(z)=\zeta\right\}.
\end{equation}	
It is not obvious, a priori, whether $\psi_0=\phi_0^*$ or not. It will become evident by the end of the proof, but at the moment we have to treat this two functions as separate objects.

Owing to Fenchel's inequality (Lemma \ref{lemma:convprop}-(ii)),  $\forall\,\lambda \in \rm$, $\forall\,u\in \vpot$, $\forall\, s \in \lpsol$,
\begin{equation*}
	\lambda \cdot E(s) = \int_\O (\lambda +u(\o))\cdot s(\o)\, d\mu \leq \int_\O \phi(\lambda + u(\o),\o) + \phi^*(s(\o),\o)\, d\mu.
\end{equation*}
Passing to the infimum over $u\in \vpot$ and recalling definitions \pref{def:phi0} and \pref{def:psi0} we get
\begin{equation}
\label{eq:liminf}
	\lambda \cdot E(s)  \leq \phi_0(\lambda) + \psi_0(E(s)),
\end{equation}
and passing to the supremum on $\lambda \in \rm$ we read
\begin{equation}
\label{eq:psiinf}
	\phi^*(E(s))=\sup_{\lambda \in \rm} \left\{ \lambda \cdot E(s)  - \phi_0(\lambda)\right\} \leq \psi_0(E(s)).
\end{equation}
Let now $v,z$ be the solutions of the auxiliary problem
\begin{equation}
\label{eq:auxrecall}
	v\in \vpot,\quad z \in \lpsol\qquad \mbox{with}\quad z\in \partial \phi(\xi + v),
\end{equation}
for $\xi \in \rm$. By Fenchel's equality (Lemma \ref{lemma:convprop}-(v))
\begin{equation*}
	\xi \cdot E(z) =\int_\O (\xi +v(\o))\cdot z(\o)\, d\mu = \int_\O \phi(\xi + v(\o),\o) + \phi^*(z(\o),\o)\, d\mu,
\end{equation*}
and therefore, by definitions \pref{def:phi0} and \pref{def:psi0}, we get
\begin{equation}
\label{eq:limsup}
	\xi \cdot E(z)  \geq \phi_0(\xi) + \psi_0(E(z)).
\end{equation}
Combining \pref{eq:liminf} and \pref{eq:limsup} we obtain that for all $z$ which solve \pref{eq:auxrecall}
\begin{equation}
\label{eq:psieq}
	\psi_0(E(z))=\phi_0^*(E(z)).
\end{equation}	
Combining \pref{eq:psiinf}, \pref{eq:limsup}, and \pref{eq:psieq}, we obtain that for all $\xi\in \rm$, for all $(v,z)$ satisfying \pref{eq:auxrecall},
$$ \xi \cdot E(z)  = \phi_0(\xi) + \phi_0^*(E(z)), $$
and therefore, by Fenchel's equality, that
$$ E(z) \in \partial \phi_0(\xi).$$
This concludes the proof of \pref{eq:equalityhomo} and Step II.
\end{proof}
Additionally, we note that since for all $\lambda\in \rm$ there exist $\bar z \in \lpsol$, $\bar u \in \lpot$ such that
$$ E(\bar z)=\lambda,\qquad \bar z \in \partial \phi(\bar u), $$
equality \pref{eq:psieq} is valid for all $\lambda \in \rm$, and therefore $\psi_0=\phi_0^*$.

\paragraph{Step III.} Convergence. We recall that an operator $A:D(A)\subset X \to \mathcal P(X)$, defined on a Hilbert space $X$, is said to be \textit{monotone} if
$$ \left< y_2 -y_1,x_2 -x_1\right>\geq 0\qquad \forall\,x_i \in D(A),\quad \forall\, y_i \in A(x_i).$$
A monotone operator is said to be \textit{maximal monotone} if its graph is maximal among all monotone sets, in the sense of set inclusions. We are going to use the following two fundamental properties.
\begin{itemize}
	\item The subdifferential of a convex function $\phi:\rm \to \R$ is a maximal monotone operator.
	\item Let $\bar x,\, \bar y \in \rm$, if $\partial \phi$ is maximal monotone, then
\begin{equation}
\label{eq:maximal}
	\Big\{ \forall\, x\in \rm,\ \exists\, y \in \partial \phi(x),\quad \left< \bar y -y, \bar x -x\right>\geq 0\Big \}\quad  \Rightarrow\quad \bar y \in \partial \phi(\bar x).
\end{equation}	
\end{itemize}

Let now $\xi,v,z$ be as in \pref{eq:auxrecall}, we define the oscillating test functions
$$ v_\n(x):= x \mapsto v\left( T_{x/\n}\o \right),\qquad z_\n(x):= x \mapsto z\left( T_{x/\n}\o \right),$$
$$ \phi_\n(\cdot,x)=\phi(\cdot, T_{x/\n}\o),$$
which satisfy
\begin{subequations}
	\begin{eqnarray}
		&& v_\n \in L^p_\textup{pot}\left(\rm;\rm\right)\label{eq:gradv},\\
		&& z_\n \in L^{p'}_\textup{sol}\left(\rm;\rm\right)\label{eq:divz},\\
		&& z_\n \in \partial\phi_\n(\xi + v_\n).\label{eq:zetainpsi}
	\end{eqnarray}
\end{subequations}
We can find a set $N\subset \O$, with $\mu(N)=0$, such that conditions \pref{hyp:conv} and \pref{hyp:growth} are satisfied for all $\o\in \O\backslash N$. Fix $\o \in \O\backslash N$ and consider now the problem
\begin{align}
	-\textup{div}\, \sn &= f^\n & 			&\mbox{weakly in } Q,\label{eq:divsn}\\
	\sn& \in \partial \phi_\n (\nabla \un) &	&\mbox{a.e. in } Q, \label{eq:subsn}\\
	\un &=0 & 						&\mbox{a.e. on }\partial Q.
\end{align}	
As in the proof of Step II above, $(\sn,\un)$ is a solution of \pref{eq:divsn}-\pref{eq:subsn} if and only if $\un$ realizes
\begin{equation}
\label{eq:minueta}
    \min \left\{ \int_Q \phi_\n(\nabla u(x),x) -f^\n(x)u(x)\, dx,\quad u \in W^{1,p}_0(Q) \right\}
\end{equation}    
and $\sn$ realizes
\begin{equation}
\label{eq:minseta}
    \min \left\{ \int_Q \phi^*_\n( s(x),x) \, dx,\quad s \in L^{p'}(Q;\rm),\ -\div\, s = f^\n \right\}.
\end{equation}    
By direct methods of calculus of variations, as in the proof of Step II, for all $\n>0$ there exists a (not necessarily unique) solution $(\un,\sn)$. Denoting by $_p\langle\, ,\, \rangle_{p'}$ the duality between $W^{1,p}_0(Q)$ and $\W(Q)$, by Lemma \ref{lemma:convprop} -(v), \pref{hyp:growth}, and Lemma \ref{lemma:stargr}
\begin{align*}
	\min\{c_1,\bar c_1\}&\left({\|\nabla \un\|}^p_{L_{\phantom 0}^p(Q)}  +  {\|\sn\|}^{p'}_{L_{\phantom 0}^{p'}\!(Q;\rm)}\right) - c_0|Q| \\
	&\leq \int_Q \phi_\n(\nabla \un) + \phi^*_\n(\sn)\, dx 
	= \int_Q \nabla \un \cdot \sn\, dx=\ _p \langle u\, ,\,f^\n\rangle_{p'}.
\end{align*}	
Since $(f^\n)$ is a converging sequence, we can assume that $\|f^\n\|_{\W\!(Q)}$ is bounded by a constant $C_f$, and by Poincar\'e's lemma we can find a constant $C>0$, depending on $Q,C_f,c_0,c_1,\bar c_1$, but independent of $\n$, such that
$$ {\|\un\|}_{W_{\phantom 0}^{1,p}(Q)}\leq C,\qquad {\|\sn\|}_{L_{\phantom 0}^{p'}\!(Q;\rm)}\leq C.$$
By compactness, for all $\o \in \O\backslash N$ we can find a subsequence $(\n_k)$ and a limit point $(u^*(\o),\s^*(\o))=(u^*,\s^*)\in W^{1,p}_0(Q)\times L_{\phantom 0}^{p'}(Q;\rm)$ such that 
\begin{equation}
\label{eq:convnu}
	u^\n(\cdot,\o) \weakto u^* \qquad \mbox{weakly in }W^{1,p}(Q)
\end{equation}	
and
\begin{equation}
\label{eq:convsn}
	\s^\n(\cdot,\o) \weakto \s^* \qquad \mbox{weakly in }L_{\phantom 0}^{p'}(Q).
\end{equation}	
Testing equation \pref{eq:divsn} with $g\in W^{1,p}_0(Q)$ we have
$$ \int_Q \sn \cdot \nabla g\, dx =\, _p \langle g\, ,\,f^\n\rangle_{p'}, $$
and passing to the limit as $\n \to 0$ we find
$$ \int_Q \s^* \cdot \nabla g\, dx =\, _p \langle g\, ,\,f\rangle_{p'}, $$
that is,
$$ -\div\, \s^* = f.$$
It remains to prove that
$$ \s^* \in \partial \phi_0(\nabla u^*).$$
By monotonicity of $\partial \phi_\n$ and equations \pref{eq:subsn} and \pref{eq:zetainpsi}, $\forall\, \n>0$, $\forall\, \xi \in \rm$, for a.e. $x\in Q$, and $\mu$-a.e. $\o \in \O$ it holds
\begin{equation}
\label{eq:monot}
	( \s^\n - z_\n)\cdot(\nabla u^\n -(\xi +v_\n)) \geq 0.
\end{equation}	
By Birkhoff's ergodicity Theorem \ref{th:birkhoff} and the definition of $v_\n$ and $z_\n$ 
\begin{equation}
\label{eq:convbirk} 
	z^\n \weakto E(z)\quad \mbox{weakly in }L_{\phantom 0}^{p'}(Q)\qquad \mbox{and}\qquad \xi +v_\n \weakto \xi\quad \mbox{weakly in }L_{\phantom 0}^p(Q).
\end{equation}	
By \pref{eq:divsn} and \pref{eq:divz} 
$$ \left\{\div\, \s^\n\right\}_\n,\ \left\{\div\, z_\n\right\}_\n\qquad \mbox{are compact sets in }\W(Q),$$
and by \pref{eq:gradv} $\curl\, v_\n =0$. Therefore the hypothesis of the \textit{div-curl} Lemma \ref{lemma:divcurl} are satisfied, and by convergences \pref{eq:convnu}, \pref{eq:convsn}, and \pref{eq:convbirk} we can pass to the limit as $\n \to 0$, in \pref{eq:monot}, obtaining
$$ (\s^* - E(z))\cdot(\nabla u^* - \xi)\geq 0\qquad \mbox{for a.e. }x\in Q.$$
Owing to Step II, we know that $E(z)\in \partial \phi_0(\xi)$, and by the arbitrariety of $\xi\in \rm$ and the maximality property \pref{eq:maximal} we obtain that $ \s^* \in \partial \phi_0(\nabla u^*)$ for a.e. $x\in \rm$. This concludes the proof of Theorem \ref{th:nonlinear}.
\end{proof}

\addcontentsline{toc}{section}{Bibliography}
%
\def\cprime{$'$}

\end{document}